\documentclass[11pt,a4paper]{article}

\usepackage{pdfsync}
\usepackage{graphicx}
\usepackage{latexsym}
\usepackage[small,bf]{caption}
\usepackage{amsfonts}
\usepackage{amssymb}
\usepackage{amsmath}
\usepackage{chemarr}
\usepackage{stmaryrd}
\usepackage{bbm}
\usepackage{times}

\usepackage{hyperref}
\usepackage{setspace}
\usepackage{enumitem}
\usepackage{authblk}

\usepackage[top=3cm, bottom=3cm, outer=3cm, inner=3cm, heightrounded, marginparwidth=1.9cm, marginparsep=0.3cm]{geometry}
\usepackage{amsthm}
\makeatother
\newtheorem{thm}{Theorem}

\newtheorem{prop}{Proposition}
\newtheorem{cor}{Corollary}
\theoremstyle{definition}

\newtheorem{remark}{Remark}

\newcommand{\N}{\mathbb{N}}
\newcommand{\R}{\mathbb{R}}
\newcommand{\E}{\mathbb{E}}

\renewcommand{\P}{\mathbb{P}}

\newcommand{\Lcal}{\mathcal{L}}

\usepackage{bbm}
\usepackage{enumerate}
\usepackage{xcolor}
\usepackage[normalem]{ulem}
\usepackage{xfrac}
\usepackage{multirow}

\author{Angeliki Koutsimpela\footnote{ angeliki.koutsimpela@uni-a.de}}
\affil{University of Augsburg}

\date{}                                  

\newcommand{\veps}{\varepsilon} 
\newcommand{\W}{W}

\renewcommand{\P}{\mathbb{P}}

\begin{document}

\title{Mean-field limits in interacting particle systems with symmetric superlinear rates} 

\maketitle

\begin{abstract}We derive the exchange-driven growth (EDG) equations as the mean-field limit of  interacting particle systems on the complete graph. Extending previous work, we consider symmetric exchange kernels $c(k, l) = c(l, k)$ satisfying super-linear bounds of the form $c(k, l) \leq C(k^{\mu}l^{\nu} + k^{\nu}l^{\mu})$, with $0 \leq \mu, \nu \leq 2$ and $\mu + \nu \leq 3$. Under these conditions the EDG equations are known to have global solutions. We establish a law of large numbers for the empirical measures, showing that the solutions of the EDG equation describe the limiting distribution of cluster sizes in the particle system. Furthermore, we analyse the dynamics of tagged particles and prove convergence to a time-inhomogeneous Markov process governed by a nonlinear master equation, derived via a law of large numbers for size-biased empirical processes.
\end{abstract}


\section{Introduction}

In applied sciences, aggregation phenomena are often modelled by particle systems driven by exchange mechanisms between clusters. 
In the exchange-driven growth (EDG) model, clusters are composed of an integer number of monomers and interact by exchanging a single monomer at a time \cite{ben2003exchange}. Denoting by \(\{k\}\) a cluster of integer size \(k\), the exchange mechanism in EDG can be formally represented as:
\begin{equation}\label{chem}
\{k\} + \{l-1\} \xrightleftharpoons[c(l,k-1)]{c(k,l-1)} \{k-1\} + \{l\}\ ,\quad k,l\geq 1\ ,
\end{equation}
where the kernel \(c(k,l-1)\geq 0\) denotes the rate at which a monomer is transferred from a cluster of size  $k$ to one of size $l-1$.
This exchange mechanism is fundamentally different from those in coagulation models, such as the well-studied Smoluchowski equation \cite{Smoluchowski1927}. The EDG model applies to a broad range of phenomena across both the natural and social sciences, including problems in kinetic theory, ferromagnetism \cite{Krapivsky2010}, cloud and galaxy formation \cite{Hidy1972}, as well as wealth exchange \cite{Ispolatov1998} and migration \cite{Leyvraz2002}.
Mathematically, if \(f_k(t)\) denotes the fraction of clusters containing $k\in \N$ particles at time $t\in [0,\infty )$, the exchange mechanism in \eqref{chem} leads to the following mean-field rate equation
\begin{align}\label{EDG}
\frac{df_{k}(t)}{dt}
&=\sum_{l\geq 1}c(k+1,l)f_l(t)f_{k+1}(t)+\sum_{l\geq 1} c(l,k-1) f_l(t)f_{k-1}(t) \nonumber \\
&\quad-\left(\sum_{l\geq 0} c(k,l)f_l(t)+ \sum_{l\geq 0} c(l,k)f_l(t)\right)f_{k}(t)\quad \text{for all }k\geq 0,
\end{align}
where we set $f_{-1} (t)\equiv 0$ and naturally $c(0,l)= 0$ for all $l\geq 0$. 
Questions regarding the well-posedness of equation \eqref{EDG} have attracted considerable attention, with various assumptions imposed on the kernels  $c(k,l)$ (see e.g. \cite{Eichenberg04032021,  FLeyvraz_1981, globalEDG, WWhite1980}). In the physics literature, equations of the form \eqref{EDG} are typically viewed as scaling limits of stochastic interacting particle systems in the limit of diverging system size. This connection was rigorously established in \cite{grosskinsky2019derivation}, where a law of large numbers for the empirical measures was proven for misanthrope-type processes \cite{cocozza1985processus} on the complete graph, under the assumption that the jump rates are bounded by a bilinear function, i.e., 
$c(k,l)\leq Ck(1+l)$, for some constant $C>0.$ In that framework, the solution $f_k(t)$ to \eqref{EDG} asymptotically describes the fraction of sites occupied by $k$ particles, and a propagation of chaos result was established. That is, the dynamics on any finite number of sites in the particle system become asymptotically independent birth-death processes with master equation given by \eqref{EDG}. In \cite{schlichting2024variational} this mean-field scaling limit has been studied in the context of variational convergence for gradient flow structures. Another similar result has been derived recently for a system where more than one monomer can be exchanged and masses are rescaled, leading to a generalized EDG system with continuous mass index \cite{lam2025convergence}. In both works the above bilinear bound on the jump rates is assumed.

In this work, we extend the results of \cite{grosskinsky2019derivation} by focusing on misanthrope-type processes with symmetric kernels 
\(c(k,l)=c(l,k)\) (exchange is directionally neutral), but allowing for super-linear growth bounds of the form $$c(k,l)\leq C(k^{\mu}l^{\nu}+k^{\nu}l^{\mu})$$
with $0\leq \mu,\; \nu\leq 2$ and $\mu+\nu\leq 3$. 
In particular, we prove a law of large numbers for the empirical measures of the underlying particle system. Under the symmetry assumption on the rates, equation \eqref{EDG} now corresponds to the master equation of a symmetric random walk absorbed at zero (see equation \eqref{mastereq}). 
Moreover, we investigate the behaviour of tagged particles, extending the results of \cite{SGAKtagged} within the assumptions of our model. Specifically, we show that in the mean-field limit, the occupation number at a tagged particle's site converges to a time-inhomogeneous Markov process, whose dynamics are governed by a nonlinear master equation derived from a law of large numbers for size-biased empirical processes.

It is worth noting that global well-posedness of equation \eqref{EDG} under symmetric kernels satisfying the above bounds has been independently studied and established in \cite{esenturk2018mathematical, globalEDG}. However, for kernels of higher order, new challenges emerge: solutions can blow up in finite time, and in some cases even exhibit instantaneous gelation (i.e., blow-up at time zero) \cite{ Eichenberg04032021, esenturk2018mathematical}. A rigorous investigation into the connection between the loss of global well-posedness for equation \eqref{EDG} and the microscopic particle system along the heuristics in \cite{ben2003exchange} lies beyond the scope of the present study and will be addressed in future work.

\section{Notation and a Law of Large Numbers \label{sec:notation}}

\subsection{Mathematical setting}

We consider stochastic particle systems $({\eta}(t):t>0)$ on finite lattices $\Lambda$ of size $|\Lambda|=L\geq 2$. Configurations are denoted by ${\eta} =(\eta_x :x\in\Lambda)$ where $\eta_x \in\mathbb{\N}_0$ is the number of particles on site $x$. We consider systems with a fixed number of particles $N=\sum_{x\in\Lambda} \eta_x$ and the state space of all such configurations is denoted by $E_{L,N}\subset \N_0^\Lambda$. 
The dynamics of the process is defined by the infinitesimal generator 
\begin{equation}
	\label{eq:GenMis}
	(\mathcal{L}g)({\eta})=\sum_{x,y\in\Lambda}q(x,y)c(\eta_{x},\eta_{y})(g({\eta}^{x\rightarrow y})-g({\eta}))\ ,\quad g\in C_b (E_{L,N})\ .
\end{equation}
Here, the usual notation ${\eta}^{x\rightarrow y}$ indicates a configuration where one particle has moved from site $x$ to $y$, i.e. $\eta_{z}^{x\rightarrow y}=\eta_{z}-\delta_{z,x} +\delta_{z,y}$, and $\delta$ is the Kronecker delta.  To ensure that the process is non-degenerate, the jump rates satisfy
\begin{equation}\label{cassum}
	\left\{ \begin{array}{cl}
		c(0,l)=0\;&\mbox{for all }l\geq 0\\
		c(k,l)>0\;&\mbox{for all }k>0\;\mbox{and }l> 0.
	\end{array} \right.
\end{equation}
Our main further assumptions on the rates are that they are symmetric
\begin{equation}\label{symmetry}
    c(k,l)=c(l,k)\quad\mbox{for all }k,l\geq 0\ ,
\end{equation}
and satisfy a superlinear bound
\begin{equation}\label{kernel}
    c(k,l)\leq C(k^{\mu}l^{\nu}+k^{\nu}l^{\mu})\quad\mbox{for all }k,l\geq 0\mbox{ and some }C>0\ ,
\end{equation}
where $0\leq\mu,\nu \leq 2$ and $\mu+\nu\leq 3$. Since $E_{L,N}$ is finite, the generator \eqref{eq:GenMis} is defined for all bounded, continuous test functions $g\in C_{b}(E_{L,N})$. Symmetry \eqref{symmetry} and \eqref{cassum} imply that $c(k,0)=0$ for all $k\geq 0$, so once a lattice site becomes empty it stays empty. 
We focus on complete graph dynamics, i.e. $q(x,y)=1/(L-1)$ for all $x \neq y$, and
\begin{equation}\label{eq:consmass}
	\sum_{x\in\Lambda} \eta_x (t)\equiv N\quad\mbox{is the only conserved quantity}\ .
\end{equation}
The system eventually converges to one of $L$ absorbing states, where all particles are contained in a single cluster.

In the following, we denote by $\mathbb{P}^L$ and $\mathbb{E}^L$ the law and expectation on the path space $\Omega=D_{[0,\infty)}(E_{L,N})$ of the process $\big( \eta(t): \;t\geq 0\big)$. 
As usual, we use the Borel $\sigma$-algebra for the discrete product topology on $E_{L,N}$, and the smallest $\sigma$-algebra on $\Omega$ such that $\omega\mapsto \eta_t(\omega)$ is measurable for all $t\geq 0$. 
We consider the empirical processes $t\mapsto F_k^L ({\eta}(t))$ with
\begin{equation}
	F_{k}^L ({\eta}):=\frac{1}{L}\sum_{x\in\Lambda}\delta_{\eta_{x},k} \in [0,1]\ ,\quad k\geq 0\ , \label{fk}
\end{equation}
counting the fraction of lattice sites for each occupation number $k\geq 0$. 
For our main result we will consider the thermodynamic limit with density $\rho$, i.e.
\begin{equation}\label{thermo}
	L\to\infty ,\ N=N_L \to\infty \quad\mbox{such that}\quad N/L\to\rho\geq 0\ .
\end{equation} 
Under condition \eqref{thermo}, the sequence $N/L$ is bounded from above by a constant and for simplicity and without loss of generality, we assume that
\begin{equation}\label{thermob}
	 N/L\leq \rho\quad\mbox{for all }L\geq 2\ .
\end{equation}

For the sequence (in $L$) of initial conditions $(\eta(0) ,X(0))$ we first require the minimal condition that there exists a fixed probability distribution $f (0)$ on $\N_0$ with finite moments
\begin{equation}\label{initialcon0}
	m_1(0) := \sum_{k} kf_k(0) =\rho < \infty  \quad\text{and}\quad   m_4(0) := \sum_{k\geq 1}k^4 f_k(0) < \infty,
\end{equation} 
such that we have a weak law of large numbers
\begin{equation}\label{initialcon0b}
	F_k^L (\eta(0)) \stackrel{d}{\longrightarrow} f_k (0) \quad\text{as }L\to\infty ,\ \text{for all } k \geq 0.
\end{equation}
As a further regularity assumption on the initial conditions we require a uniform bound for the third moment, i.e.\ for some fixed $\alpha >0$
\begin{align}
	\label{initialcon0c}
	 \E^L \Big[\frac{1}{L}\sum_{x \in \Lambda}\eta_x^3(0)\Big]\leq \alpha \quad\mbox{for all }L\geq 2\ .
\end{align}
Note that (\ref{thermob}) and conservation of mass \eqref{eq:consmass} imply for the first moment that
\begin{equation}
	\label{alpha}
	\frac{1}{L}\sum_{x\in\Lambda} \eta_x (t)=\sum_{k\geq 0}k F_k^L (\eta (t))=\frac{N}{L}\leq \rho \ , \quad\P^L -a.s.\ \mbox{for all }t\geq 0\mbox{ and }L\geq 2\ .
\end{equation}

\subsection{A law of large numbers for empirical processes}


\begin{thm}
	\label{thmfactorization}
	Consider a process with generator \eqref{eq:GenMis} on the complete graph with a symmetric kernel $c(k,l)$ satisfying the bound \eqref{kernel} and initial conditions satisfying \eqref{initialcon0}, \eqref{initialcon0b} and \eqref{initialcon0c}.  
	Then we have in the thermodynamic limit \eqref{thermo} for any $\rho >0$ and any bounded function $h:\N_0\to\R$,
	\begin{equation}
		\Big(\sum_{k\geq 0} F_k^L (\eta(t) )\, h(k) :t \geq0\Big)\to \Big(\sum_{k\geq 0} f_k (t)\, h(k) :t\geq0\Big)\quad\mbox{weakly on }    {D_{[0,\infty)}(\R )}\ .
	\end{equation}
	Here 
	$t\mapsto f(t)=(f_k(t):k\in \mathbb{N}_0)$ is the unique global solution to the \textbf{mean-field equation}
	\begin{align}\label{mastereq}
\frac{df_{k}(t)}{dt} = \mu_{k+1} (t) \, f_{k+1}(t)+\mu_{k-1} (t) \, f_{k-1}(t) -2\mu_{k} (t)  f_{k}(t)\ ,\quad k\geq 0,
	\end{align}
    where \begin{equation}\label{birthrate}
\mu_k(t):=\sum \limits_{l\geq 1}c(k,l)f_l(t)\end{equation} 
	with initial condition $f(0)$ given by (\ref{initialcon0b}). 
	Here we use the convention $f_{-1}(t)\equiv 0$ for all $t\geq 0$.
\end{thm}

Following \cite{grosskinsky2019derivation}, if we further assume symmetry on the initial conditions, i.e. for all $L\geq 2$,
\begin{equation}\label{sym}
    \text{the distribution of } \{\eta_x(0):  x\in \Lambda\} \text{ is invariant under permutation of lattice sites,}
\end{equation}
then by symmetry of the dynamics, this holds also for the whole process $(\eta(t): \; t\geq 0)$ and 
the law of large numbers of Theorem \ref{thmfactorization} implies propagation of chaos in the following sense.

\begin{cor}\label{poc}
 Consider the process with generator \eqref{eq:GenMis} and conditions as in Theorem \ref{thmfactorization} together
with \eqref{sym}. Propagation of chaos holds, i.e. for all $m \geq 1$ as $L \to \infty$, the finite
dimensional processes
$((\eta_1(t), \eta_2(t), . . . , \eta_m(t)) :\; t\geq 0)$ converge weakly on path space $D_{[0,\infty )} (\N_0^m )$ to
independent, identical birth-death processes on $\N_0$ with distribution $f(t)$ and (nonlinear) master equation given by \eqref{mastereq}.   
\end{cor}

This is a standard consequence of the law of large numbers in Theorem \ref{thmfactorization} (see e.g.  \cite{sznitman1991topics, daipra2017}) and an exposition of a proof on the level of mean-field particle systems can be found in \cite{grosskinsky2019derivation}. 
Notice that \eqref{mastereq} can be identified as the master equation of a non-linear symmetric random walk on $\N_0$ with absorbing boundary at $k=0$ and with time-dependent rate at state $k$ equal to $\mu_k (t)$, as given in \eqref{birthrate}. According to Corollary \ref{poc}, this corresponds to the limiting dynamics of the occupation number for a fixed number of sites which evolve independently.

\section{Proof of the main result\label{sec:proof}}

\subsection{Moment bounds}

As a first step we collect some useful results on moments and establish a time-dependent bound on the moments of the processes $\eta_x(t)$ for $x\in\Lambda$. 
For any integer $n\geq 0$ denote the $n$-th moment by
\begin{equation}\label{lmoment}
	m_n^L (t):=\E^L \Big[\frac{1}{L}\sum_{x\in\Lambda} \big(\eta_x (t)\big)^n\Big] =\E^L \Big[\sum_{k\geq 0} k^n F_k^L (\eta (t))\Big]\ .
\end{equation}
We have $m_0^L (t)\equiv 1$ and with \eqref{initialcon0b}, $m_1^L (0)\to\rho$ and $m_2^L (0)\to m_2 (0)<\infty$. The uniform conditions \eqref{initialcon0c} on the moments further imply for all $L\geq 2$ that $m_2^L (0)\leq m_3^L (0)\leq \alpha,$ and with conservation of mass \eqref{alpha}, we have $m_1^L (t)\leq \rho$ for all $t\geq 0$, while higher moments typically grow in time for condensing systems (see e.g. \cite{godreche2016coarsening,jatuviriyapornchai2016coarsening,schlichting2020exchange}). The following result gives a general (but very rough) upper bound.

\begin{prop}\label{gwall}
	Assume the conditions of Theorem \ref{thmfactorization} and that   $m^L_{n}(0)<\infty$ for some integer $n\geq 2$ and for all $L\geq 2$. Then there exists a constant $  {C}_n>0$ independent of $L$ such that for all $L\geq 2, \; t\geq 0 $
    \begin{equation}
        m^L_{n} (t) \leq   m^L_{n}(0) e^{   C_nt}  
    \end{equation}
    Furthermore, if $n>\nu$, $\mu>1$ with $\mu ,\nu\in [0,2]$ defined in \eqref{kernel}, we have the polynomial bound
\begin{equation}
m_n^L(t)\leq \left(m_n^L(0)^{\frac{3-(\mu+\nu)}{n}} + 2C_n\rho\frac{3-(\mu+\nu)}{n} t \right)^{\frac{n}{3-(\mu+\nu)}} \ .
\end{equation}
\end{prop}
\begin{proof}
	Applying the generator \eqref{eq:GenMis} to the function $g(\eta)=\eta_x^{n}$ and some $x\in \Lambda$, we get
	\begin{equation}\label{genhm}
		\Lcal \eta_x^{n}=\frac{1}{L-1}\sum_{y\neq x} c(\eta_x,\eta_y)\left((\eta_x+1)^n+(\eta_x-1)^n-2\eta_x^n\right)=\frac{1}{L-1}\sum_{y\neq x} c(\eta_x,\eta_y) p_{n-2}(\eta_x)   {\ .}    
	\end{equation}
    where $p_{n-2}$ is a polynomial of degree $n-2$.
	Therefore, using that $k\mapsto\eta_x^{k}$ is non-decreasing and the bound \eqref{kernel} we find for some positive constant $C_n$ (independent of $L$)
	\begin{align*}
		\frac{d}{dt}m^L_n(t)  & = \frac{1}{L} \sum \limits_{x\in \Lambda} \E [\mathcal{L}\eta^n_x(t)] \\ & \leq C_n  \E^L\left[ \sum \limits_{k}k^{\mu+n-2}F^L_k(\eta(t)) \sum \limits_{l}l^{\nu}F^L_l(\eta(t))\right] \\&  \qquad\qquad {+} C_n\E^L\left[ \sum \limits_{k}k^{\nu+n-2}F^L_k(\eta(t)) \sum \limits_{l}l^{\mu}F^L_l(\eta(t))\right]  \\& \leq  C_n\rho  \left(\E^L\left[ \sum \limits_{k}k^{p_1(\mu+n-3)+1}F^L_k(\eta(t)) \right] \right)^{\frac{1}{p_1}}  \left(\E^L\left[\sum \limits_{l}l^{q_1(\nu-1)+1}F^L_l(\eta(t))\right]\right)^{\frac{1}{q_1}} \\&  + C_n\rho  \left(\E^L\left[ \sum \limits_{k}k^{p_2(\nu+n-3)+1}F^L_k(\eta(t)) \right] \right)^{\frac{1}{p_2}}  \left(\E^L\left[\sum \limits_{l}l^{q_2(\mu-1)+1}F^L_l(\eta(t))\right]\right)^{\frac{1}{q_2}}
	\end{align*}
using Hölder's inequality with respect to the size-biased measures $k\, F_k (\eta (t)) L/N$. We want to choose $p_i,q_i\in [1,\infty]$, $i=1,2$ such that $\frac{1}{p_i}+\frac{1}{q_i}=1$ for $i=1,2$ and 
    \begin{equation}\label{pqres}
        \begin{cases}
        p_1(\mu+n-3)+1\leq n \\
        q_1(\nu-1)+1 \leq n
        \end{cases} \; \text{ and } \; \begin{cases}
        p_2(\nu+n-3)+1\leq n \\
        q_2(\mu-1)+1 \leq n
        \end{cases}
    \end{equation}
To satisfy \eqref{pqres}, we select  $p_1,q_1$ according to the following table:\\

\begin{center}
\begin{tabular}{|c|c|}
\hline
 (i)  $\nu\leq 1$ or $n=\nu=2$  & (ii) $ n>\nu>1$ \\\hline
 $p_1=\infty, q_1=1$ & {$p_1=\frac{n-1}{n-\nu},\; q_1=\frac{n-1}{\nu-1}$}  
\\ 
\hline
\end{tabular}
\end{center}
Similarly, we select $p_2,q_2$ according to the following table:\\

\begin{center}
\begin{tabular}{|c|c|}
\hline
  (iii) $\mu\leq 1$ or $n=\mu=2$  & (iv) $n>\mu>1$   \\\hline
 {$p_2=\infty, q_2=1$}  & {$p_2=\frac{n-1}{n-\mu} , q_2=\frac{n-1}{\mu-1}$}
\\ 
\hline
\end{tabular}
\end{center}
Therefore, we find $\frac{d}{dt}m^L_n (t)\leq 2C_n\rho m^L_n(t)$ which by Gronwall's inequality implies
 \[
m^L_n(t)\leq m^L_n(0)e^{2C_n\rho t} \ . 
 \]
Notice that in case $n>\mu,\nu>1$, we find the polynomial bound:
\begin{align*}
 \frac{d}{dt}m^L_n (t)&\leq C_n \rho   (m^L_n(t))^{\frac{p_1(\mu+n-3)+1}{np_1}}(m^L_n(t))^{\frac{q_1(\nu-1)+1}{nq_1}}  \\& \qquad +C_n\rho (m^L_n(t))^{\frac{p_2(\mu+n-3)+1}{np_2}}(m^L_n(t))^{\frac{q_2(\nu-1)+1}{nq_2}}  \\& =  2 C_n\rho  (m^L_n(t))^{1-\frac{3-(\mu+\nu)}{n}}
 \end{align*}
which implies for \(\mu+\nu<3\)
\[
m_n^L(t)\leq \left(m_n^L(0)^{\frac{3-(\mu+\nu)}{n}} + 2C_n\rho\frac{3-(\mu+\nu)}{n} t \right)^{\frac{n}{3-(\mu+\nu)}} \ .
\]
\end{proof}

Assuming a uniform in $L$ bound on the initial third moment,  we have that the third moment remains bounded uniformly in $L$ in every compact time domain. 

\begin{cor}\label{mom3}
    Under the assumptions of Theorem \ref{thmfactorization}, for all $T>0$ there exists $D_T>0$ (independent of $L$) such that 
    \begin{equation}
		m^L_3(t)\leq D_T  \quad\mbox{ for all }  L\geq 2, \; t \in [0,T] 
    \end{equation} 
\end{cor}
Recall that due to symmetry of the kernel $c(k,l)$, the process \eqref{eq:GenMis} on the complete graph has $L$ absorbing states, corresponding to the cases where all the particles accumulate on one of the $L$ sites of the lattice. It is immediate that the uniform in $L$ bound on the second moment of the process $\eta$ implies that the time to absorption converges to infinity as $L\to \infty.$

\begin{cor}
 Let $T_L=\inf \{ t\geq 0: \eta_x(t)=N \text{ for some } x\in\Lambda  \}$. Then $T_L \overset{d}{\longrightarrow} \infty$.   
\end{cor}
\begin{proof}
With Markov's inequality we have for all $t\geq 0$
    \begin{align*}
    \P^L \left( T_L\leq t \right) &= \P^L  \left(  \bigcup \limits_{x\in \Lambda} \{ \eta_x(t)=N \}\right) \leq \sum \limits_{x\in \Lambda} \P^L  \left( \eta_x(t)\geq N \right) \\&\leq \frac{\sum \limits_{x\in \Lambda} \E^L \left[ \eta^2_x(t) \right]}{N^2} = \frac{L}{N} \frac{m^L_2(t)}{N}\leq \frac{L}{N} \frac{D_t}{N} \to 0 \text{ as } L,N\to \infty \ ,
    \end{align*}
    where in the last line we used Corollary \ref{mom3} and $m_2^L (t)\leq m_3^L (t)$.
\end{proof}

\subsection{Tightness}

\begin{prop}\label{tightness}
	Consider the process with generator \eqref{eq:GenMis} and conditions as in Theorem \ref{thmfactorization}. Denote by $\mathbb{Q}^L_h$ the law of the process $t \mapsto H(\eta(t)):=\langle F^L(\eta(t)),h \rangle$ on path space $D_{[0,\infty)}(\R)$, which is the image measure of $ \mathbb{P}^L$ under the mapping $\eta\mapsto \langle F^L(\eta),h \rangle$. Then $\mathbb{Q}^L_h$ is tight as $L \to \infty$.
\end{prop}
\begin{proof}
    Using a version of Aldous' criterion to establish tightness for $\mathbb{Q}_h^L$ (cf. Theorem 16.10 in \cite{billingsley2013convergence}), it suffices to show that for all $t\geq 0$
	\begin{equation}
		\label{eq:tight1}
		\lim_{a\to \infty}\limsup_{L\to\infty} \mathbb{P}^L \big[ |H(\eta (t))|\geq a\big] =0,
	\end{equation}
	 and that for any $\epsilon>0$, $t>0$,
	\begin{equation}
		\label{eq:tightness}
		\lim_{\delta_0\to 0^+}\limsup_{L\to \infty}\sup_{\delta\leq \delta_0}\sup_{\tau \in \mathfrak{T}_t} {\mathbb{P}}^L\big[ |H(\eta(\tau+\delta))-H(\eta(\tau))|>\epsilon\big] =0,
	\end{equation}
 where $\mathfrak{T}_t$ is the set of stopping times satisfying $\tau\leq t.$\\
Since $h$ is bounded, $\big|\big\langle F^L (\eta),h\big\rangle\big|\leq \|h\|_{\infty}$. Thus \eqref{eq:tight1} follows easily from Markov's inequality,
\[ \mathbb{P}^L \big[ |H(\eta (t))|\geq a\big]\leq \frac{\|h\|_{\infty}}{a}\quad\mbox{for all }L\geq 2\ .\]
Now fix $\delta_0> 0$, $\tau \in \mathfrak{T}_t$ and consider $\delta<\delta_0$.  By It\^{o}'s formula, we have for all $u>0$
	\begin{equation}
		\label{eq:itoetaMIM}
  H({\eta}(u+\delta))-H(\eta(u) )=\int_u^{u+\delta} \Lcal H({\eta} (s))\, ds +M_h (u+\delta)-M_h(u)\ ,
	\end{equation}
	where $(M_h (u) : u\geq 0)$ is a martingale with predictable quadratic variation given by integrating the 'carr\'e du champ' operator, i.e.  
\begin{equation}
\label{eq:mart}
\langle M_h \rangle (t)=\int_0^t \big(\mathcal{L}H^2-2H\mathcal{L}H\big) ({\eta}(s))ds\ .
\end{equation}

Using again Markov's inequality in \eqref{eq:tightness} and replacing $u$ by the bounded stopping time $\tau\leq t,$ we get the bound
	\begin{align}\label{toboundMim}
		\E^L\Big[\big| H(\eta(\tau+\delta))-H(\eta(\tau))\big|\Big]&\leq \E^L \left[\int \limits_{\tau}^{\tau+\delta}  |{\Lcal} H(\eta(s))|\,ds\right] +  \E^L \Big[ \left(M_h (\tau+\delta)-  M(\tau) \right)^2\Big]^{\frac{1}{2}} \ \nonumber \\
        &\leq \sqrt{\delta} \left( \int \limits_{0}^{t+\delta}\E^L \Big[  |{\Lcal} H(\eta(s))|^2\Big]\,ds \right)^{\frac{1}{2}} {+} \E^L \big[  \langle M_h \rangle (\tau+\delta)- \langle M_h \rangle (\tau)\big]^{\frac{1}{2}}  \ .
	\end{align}
	Here we used Cauchy-Schwarz inequality to get
    \begin{align*}
\E^L \left[\int \limits_{\tau}^{\tau+\delta}  |{\Lcal} H(\eta(s))|\,ds\right] & \leq  \left(\E^L \left[\int \limits_{\tau}^{\tau+\delta}  1\,ds\right] \right)^{1/2}  \left(\E^L \left[\int \limits_{\tau}^{\tau+\delta}  |{\Lcal} H(\eta(s))|^2\,ds\right] \right)^{1/2} \\ &
\leq  \sqrt{\delta}  \left(\E^L \left[\int \limits_{0}^{t+\delta}  |{\Lcal} H(\eta(s))|^2\,ds\right] \right)^{1/2}  \\& = \sqrt{\delta}  \left(\int \limits_{0}^{t+\delta} \E^L \left[ |{\Lcal} H(\eta(s))|^2\,ds\right] \right)^{1/2}\ ,
\end{align*}
as well as H\"older's inequality and the optional stopping theorem for the martingale $M_h^2 (t)-\langle M_h \rangle (t)$.\\
To compute $\Lcal H(\eta )$, we first recall that 
\begin{equation}
H(\eta)=\langle h, F^L({\eta}) \rangle =\sum_{k\geq 0}h(k) \frac{1}{L}\sum_{x\in\Lambda} \delta_{\eta_x,k}= \frac{1}{L}\sum_{x\in\Lambda} h(\eta_x)
\end{equation}
Therefore, $\mathcal{L} H(\eta )=\frac1L \sum_{x\in\Lambda} \mathcal{L} h(\eta_x )$ and
\begin{equation}
\mathcal{L}H({\eta})=
\frac{1}{L}\frac{1}{L-1}\sum_{x\in\Lambda}  \sum \limits_{y\neq x} c(\eta_x, \eta_y) \Big(h(\eta_x-1)-2h(\eta_x)+h(\eta_x+1)\Big) \ .
\label{compu}
\end{equation}
Thus, we have for all $\eta \in E_{L,N}$
\begin{equation*}
    |\mathcal{L}H({\eta})|\leq 16C\|h\|_{\infty} \left(\sum_{k}  k^{\mu}F^L_k(\eta) \right) \left(\sum_{l}  l^{\nu}F^L_l(\eta) \right) 
\end{equation*}
which for $\mu,\nu\leq 2$ and $\mu+\nu\leq 3$ implies that 
\begin{equation}
    \E^L\big[ |\mathcal{L}H({\eta(t)})| \big]\leq 16C\|h\|_{\infty} \rho m^L_2(t) 
\end{equation}
and
\begin{equation}\label{Lh2}
\E^L \big[\big|\mathcal{L}H({\eta}(t))\big|^2  \big]
\leq 16^2 C^2 \|h\|^2_{\infty} \rho^3  m^L_3(t) 
\end{equation}
where we used Jensen's and Hölder's inequality with respect to the probability measure $\{ \frac{kF^L_k(\eta)}{N/L} \}$ and that $N/L\leq \rho.$ 
Therefore, with $\delta\leq\delta_0$,  \eqref{toboundMim} becomes 
\begin{equation}\label{tight_bound2}
     \E^L\Big[\big| H(\eta(\tau+\delta))-H(\eta(\tau))\big|\Big] \leq \sqrt{ \delta_0 (t+\delta_0) }16C\|h\|_{\infty} \rho^{3/2} D_{t+\delta_0}^{1/2}  + \E^L \big[  \langle M_h \rangle (\tau+\delta)-  \langle M_h \rangle (\tau)\big]^{1/2} 
\end{equation}
    
Then, to control the last term of \eqref{tight_bound2}, it suffices to bound (uniformly in $L$) the 'carr\'e du champ' operator, for which we have 
\begin{align*}
    \left(\mathcal{L}H^2-2H\mathcal{L}H \right)(\eta) 
&= \frac{1}{L^2}\frac{1}{L-1}\sum_{\overset{x,y\in\Lambda}{y\neq x}} c(\eta_x, \eta_y) \big( h(\eta_x-1)-h(\eta_x) + h(\eta_y+1)-h(\eta_y)\big)^2\\
&\leq  \frac{16C\|h\|^2_{\infty}}{L-1}  \left(\sum_{k}  k^{\mu}F^L_k(\eta) \right) \left(\sum_{k}  k^{\nu}F^L_k(\eta) \right) \ .
\end{align*}
Therefore, similarly to \eqref{Lh2}, we find
\begin{equation}\label{Mh2}
   \E^L \big[ \big( \mathcal{L}H^2-2H\mathcal{L}H \big)^2 (\eta(t)) \big]  \leq \frac{16^2C^2\rho^{3}\|h\|^2_{\infty}}{(L-1)^2} m^L_3(t)  \ .
\end{equation}
 Therefore, by Cauchy-Schwarz, we find
\begin{align}\label{Mh}
\E^L \big[  \langle M_h \rangle (\tau+\delta)-  \langle M_h \rangle (\tau)\big]&=\E^L \bigg[ \int_{\tau}^{\tau+\delta} \big(\mathcal{L}H^2-2H\mathcal{L}H\big) ({\eta}(s))ds \bigg]\leq \nonumber  \\& \leq 
\sqrt{\delta_0}  \left( \E^L \bigg[ \int_{0}^{t+\delta} \Big| \big( \mathcal{L}H^2-2H\mathcal{L}H\big) ({\eta}(s))\Big|^2ds \bigg] \right)^{1/2} \nonumber  \\&\leq \sqrt{\delta_0 (t+\delta_0)}
\frac{16C\rho^{3/2}\|h\|^2_{\infty}}{L-1}  D_{t+\delta_0}^{1/2}
\end{align}
which vanishes as $\delta_0\to0^+$ (uniformly in $L$), finishing the proof.
\end{proof}

According to Prokhorov's theorem, the tightness result in Proposition 3.5 implies the existence of limit
points of the sequence
$(
\langle F^L (t), h\rangle : t \geq 0
)$
in the usual topology of weak convergence on path
space. By linearity of $h \mapsto \langle F^L (\eta(\cdot)), h\rangle$, convergence along a subsequence for a given test
function $h$ implies convergence for all bounded $h$. This establishes existence of limit processes
$( f (t) : t \geq  0)$ which may not be unique and still be random at this point.

\subsection{Characterization of the limit process}

Based on estimates \eqref{Mh2}, \eqref{Mh} in the proof of Proposition \ref{tightness}, we have that 
\begin{equation}\label{mgto0}
   \E^L\Big[ \langle M_h \rangle (t) \big] =\E^L\big[M^2_h(t)\big]\to 0 \quad \text{ as } \quad L\to \infty \quad \text{for all } \quad t\geq0\ .
\end{equation}

This implies that $M_h(t)\to 0$ as $L\to \infty$ in the $L^2$-sense. Now, writing \eqref{compu} in terms of the empirical process $F^L$ \eqref{fk}, we have
\begin{align} \label{Hlimit}
     \mathcal{L}H({\eta})&=
\frac{1}{L-1}\sum_{k\geq 1}  \sum \limits_{l\geq 1} c(k,l) F^L_k(\eta)  (F^L_l(\eta)L-\delta_{k,l})  \big( h(k-1)-2h(k)+h(k+1)\big) \nonumber \\&= \frac{L}{L-1}  \sum_{k,l\geq 1}  \big( h(k-1)-2h(k)+h(k+1)\big) c(k,l) F^L_k(\eta)   F^L_l(\eta) \nonumber\\
& \quad \quad -\frac{1}{L-1} \sum_{k\geq 1} \big( h(k-1)-2h(k)+h(k+1)\big) c(k,k) F^L_k(\eta)\ .
\end{align}

Therefore, inserting the limit points $(f(t): \; t\geq 0)$ in \eqref{Hlimit} and using \eqref{mgto0} in \eqref{eq:itoetaMIM}, they solve the following deterministic equation
\begin{equation}\label{limit}
\langle f(t),h \rangle - \langle f(0),h \rangle = \int \limits_0^t  \sum_{k\geq 1}\left(  \sum_{l\geq 1} c(k,l) f_l(\eta(s))\right)  \Big( h(k+1)-2h(k)+h(k-1)\Big) f_k(\eta(s)) \ ds
\end{equation}
for all bounded functions $h:\N_0\to \R$. Equation \eqref{limit} is equivalent to \eqref{mastereq} and has a unique global solution for symmetric kernels satisfying \eqref{kernel} when $m_4(0)<\infty$ (see \cite{esenturk2018mathematical} Theorem 6), establishing  the law of large numbers Theorem \ref{thmfactorization}. It is easy to see that in case of the symmetric product kernel $c(k,l)=(kl)^\gamma$ for $k,l\geq 1$, uniqueness can be guaranteed in a larger space, where the initial third moment (instead of the fourth) is assumed to be finite, i.e. when $m_3(0)<\infty.$

Theorem \ref{thmfactorization} and Proposition \ref{gwall} imply also convergence of moments in the following sense.

\begin{cor}\label{gammamom}
    Let $(f(t): \; t\geq 0)$ be the limit process of $(F^L(t):\; t \geq 0)$. If \( \sup_L m^L_{n+\veps}(0)<\infty \) for some \(\veps>0\), we have also convergence of moments i.e.
    \begin{equation}
       \E^L \Bigg[ \sum \limits_{k\geq 0} k^{n} |F^L_k(\eta(t)) -f_k(t)| \Bigg] \to 0 \quad \text{ as } L\to \infty \quad \text{ for all } t\geq 0\ . 
    \end{equation}
\end{cor}
\begin{proof}
   \begin{align*}
\E^L \Bigg[  \sum \limits_{k=0}^{\infty} k^{n} \big| F^L_k(\eta(t))-f_k(t)\big|\Bigg] & \leq  \E^L \Bigg[  \sum \limits_{k=0}^{\infty} k^{n+\veps} \big| F^L_k(\eta(t))-f_k(t)\big|\Bigg]^{\frac{n}{n+\veps}} \\& \qquad \qquad  \cdot \E^L \Bigg[  \sum \limits_{k=0}^{\infty} \big| F^L_k(\eta(t))-f_k(t)\big| \Bigg]^{\frac{\veps}{n+\veps}} \\ & \leq 
(2 \sup_L m^L_{n+\veps}(0) e^{C_{n+\veps}t})^{\frac{n}{n+\veps}}  \left( \E^L \Bigg[  \sum \limits_{k=0}^{\infty} \big| F^L_k(\eta(t))-f_k(t)\big|  \Bigg] \right)^{\frac{\veps}{n+\veps}}
\end{align*} 
where in the last line, we used Proposition \ref{gwall} and Fatou's Lemma.
Thus, it suffices to prove that 
\begin{equation}\label{convsum} \sum \limits_{k=0}^{\infty} \E^L \Bigg[ \big| F^L_k(\eta(t))-f_k(t)\big|  \Bigg] \to 0 \quad \text{ as } \quad L\to \infty \ . \end{equation}
For all $M>0,$ we have
\begin{align*}
  \sum \limits_{k=0}^{\infty} \E^L \Bigg[ \big| F^L_k(\eta(t))-f_k(t)\big|  \Bigg] &=   \sum \limits_{k<M} \E^L \Bigg[  \big| F^L_k(\eta(t))-f_k(t)\big|  \Bigg]   +  \sum \limits_{k\geq M} \E^L \Bigg[ \big| F^L_k(\eta(t))-f_k(t)\big|  \Bigg]   \\ & \leq   \sum \limits_{k<M}  \E^L \Bigg[ \big| F^L_k(\eta(t))-f_k(t)\big|  \Bigg]   +  \sum \limits_{k\geq M} \E^L \Bigg[  ( F^L_k(\eta(t))+ f_k(t))  \Bigg] 
\end{align*}

But 
\[
 \sum \limits_{k\geq M}  F^L_k(\eta(t)) \overset{1\leq k/M}{\leq} \frac{1}{M}
\sum \limits_{k\geq M} k F^L_k(\eta(t)) \leq \frac{\rho}{M}
\]
and the same applies also for $(f(t): \; t\geq 0)$ by Fatou's Lemma.  Therefore, based on Theorem \ref{thmfactorization}, we have for all $M>0$
$$
\limsup_{L\to \infty}  \sum \limits_{k=0}^{\infty}  \E^L \Bigg[\big| F^L_k(\eta(t))-f_k(t)\big|  \Bigg] \leq  0  + \frac{2\rho}{M} =\frac{2\rho}{M} \ . 
$$
Taking $M\to \infty$ concludes the proof.
\end{proof}

\section{Discussion: Symmetric product kernel}\label{prod}

If we consider the symmetric product kernel $c(k,l)=(kl)^{\gamma}$ for $k,l>0$, then the mean-field equation \eqref{mastereq} takes the following simple form
	\begin{equation}\label{prodmastereq}
\frac{df_{k}(t)}{dt} = m_{\gamma}(t) \Big( (k+1)^{\gamma}f_{k+1}(t)+(k-1)^{\gamma}f_{k-1}(t) -2k^{\gamma}  f_{k}(t) \Big)\ ,\quad k\geq 0,
	\end{equation}
where $m_{\gamma}(t)=\sum_k k^{\gamma}f_k(t)$ and $f_{-1} (t)\equiv 0$ for all $t\geq 0$. Provided a solution exists, recall that the moments $m_0 (t)=\sum_k f_k (t)\equiv 1$ and $m_1 (t)=\sum_k k\, f_k (t)\equiv \rho >0$ are conserved irrespectively of the parameter $\gamma >0$.

As it is reported in \cite{esenturk2018mathematical, Eichenberg04032021, ben2003exchange}, we consider the three different cases:

\begin{itemize}
    \item $0\leq \gamma \leq 3/2$: For all initial conditions with $m_3 (0)<\infty$, \eqref{prodmastereq} has a unique global solution. All moments $m_n (t)$ for $n>1$ diverge, $f_0 (t)\to 1$ and for all $k\geq 1$, $f_k (t)\to 0$ as $t\to\infty$.
    \item $3/2<\gamma\leq2$: Unique solutions exists only locally up to a gelation time
$T_{\text{gel}} < \infty$, so that $f_0 (t)\to 1$ as $t\to T_{\text{gel}}$ and again all moments with $n>1$ diverge.

\item $\gamma>2$: There is no solution for any time interval $[0,T)$, $(T>0)$ and all moments with $n>1$ explode instantaneously (instantaneous gelation).
\end{itemize}
In particular, for the average cluster size
$$\ell(t)=\frac{1}{1-f_0(t)}\sum_k kf_k(t)=\frac{\rho}{1-f_0(t)}$$
it is predicted in \cite{ben2003exchange} and confirmed in \cite{Eichenberg04032021} that we have
$$\ell(t)\varpropto \begin{cases}
    t^{\beta} &,\ \gamma<3/2\\
    \exp(Ct) &,\ \gamma=3/2\\
    (T_{\text{gel}}-t)^{\beta} &,\ 3/2<\gamma < 2
\end{cases}\ ,\quad \mbox{where}\quad \beta =(3-2\gamma )^{-1} \ . $$
In \cite{Eichenberg04032021} Theorem 1.4 it is also shown that the solutions $f_k (t)$ take a self-similar form on the coarsening scale $\ell (t)$ for $\gamma <2$.

In \cite{ben2003exchange}, the corresponding mean-field particle system with $N$ particles is analyzed heuristically. It is argued that empirical cluster size distribution converges to a solution of \eqref{prodmastereq}, and that for $\gamma >2$ the expected gelation time $T_N$ when the system reaches its absorbing state tends to $0$ as $(\log N)^{2-\gamma}$. Our result confirms the first claim on a rigorous level for all $\gamma\leq 3/2$. Note that the case $\gamma\leq 1$ was already covered by previous results \cite{grosskinsky2019derivation,schlichting2024variational,lam2025convergence}, but we are not aware of any other results for $\gamma >1$ so far. 
Convergence results for mean-field particle systems with $\gamma >3/2$ remain open for now. This will require new methods beyond controlling the empirical moments.

\section{Tagged particles}

\subsection{Notation and main result}
To follow the location $(X(t):t\geq 0)$ of a tagged particle, we extend the state space to $E:=E_{L,N} \times \Lambda$ and states $(\eta ,x)\in E$ describe the particle configuration $\eta\in E_{L,N}$ and the location $x\in\Lambda$ of the tagged particle. 
In the following, we denote again for simplicity of notation by $\mathbb{P}^L$ and $\mathbb{E}^L$ the law and expectation on the path space $\tilde{\Omega}=D_{[0,\infty)}(E)$ of the joint process $\big( (\eta(t),X(t)): \;t\geq 0\big)$. 
As before, we use the Borel $\sigma$-algebra for the discrete product topology on ${E}$, and the smallest $\sigma$-algebra on $\tilde{\Omega}$ such that $\omega\mapsto(\eta_t(\omega),X_t(\omega))$ is measurable for all $t\geq 0$. 	
The joint process is Markov and its evolution is described by the infinitesimal generator
\begin{multline}\label{taggedsyst}
	\tilde{\Lcal} G(\eta,x)= \sum_{y,z\in\Lambda}\frac{1}{L-1} c(\eta_{y},\eta_{z})(G({\eta}^{y\rightarrow z},x)-G({\eta},x))(1-\delta_{xy}) \\
	+ \sum_{z\in\Lambda} \frac{1}{L-1}c(\eta_x,\eta_z) \left[ \frac{\eta_x-1}{\eta_x}\left( G(\eta^{x\rightarrow z},x)-G(\eta,x)\right)+ \frac{1}{\eta_x}\left( G(\eta^{x\rightarrow z},z)-G(\eta,x) \right) \right]
\end{multline}
for all bounded continuous functions $G\in C_b ({E})$. 

We need further regularity assumptions on the initial conditions, namely a uniform bound of the ninth moment, for some fixed  $\alpha_9 >0$
\begin{align}
	\label{initialcon0c2}
	 \E \Big[\frac{1}{L}\sum_{x \in \Lambda}\eta_x^9(0)\Big]\leq \alpha_9 \quad\mbox{for all }L\geq 2\ .
\end{align}

We assume that $N-1$ particles are distributed on the lattice according to some initial conditions satisfying \eqref{initialcon0}, \eqref{initialcon0b}, \eqref{initialcon0c2} and the $N$-th particle (the tagged one) is located on position $X(0)$, increasing the value of $\eta_{X(0)} (0)$ by $1$ such that
\begin{equation}\label{initialcon0d}
	\E^L \left[ \eta_{X(0)}^4(0) \right]<   {\alpha}_4 \quad\text{holds for some fixed }\alpha_4 >0 \text{ and all }L\geq 2\ .
\end{equation}

For example, if we distribute $N-1$ particles uniformly, independently on $\Lambda$, \eqref{initialcon0}, \eqref{initialcon0b} are satisfied with Poisson distribution $f(0)$, and  condition \eqref{initialcon0c2} is satisfied for all $L\geq 2$. 
There are various ways to then choose the initial position of the tagged particle such that 
\eqref{initialcon0d} is satisfied. We could pick a fixed site (e.g. $X(0)=1$) or select one uniformly at random. On the other hand, selecting a site with the maximum occupation number would lead to logarithmic growth with respect to $L$ of $\eta_{X(0)}(0)$, violating \eqref{initialcon0d}.

The evolution of the occupation number on the tagged particle site is denoted by $\W^L(t):=\eta_{X(t)}(t)$. To study its dynamics, following \cite{SGAKtagged}, we apply the generator \eqref{taggedsyst} to a test function $G(\eta,x)=g(\eta_x)$
and plugging in the process \eqref{fk}, we find
\begin{multline}\label{genN}
	\hat{\Lcal}^L_{\eta(t)}  g(n)=\frac{L}{L-1}\sum_{k \geq 1} c(k,n)F^L_{k}(\eta(t)) \big(g(n+1)-g(n)\big)\\
   {+}\frac{L}{L{-}1}\bigg(\frac{n{-}1}{n} \sum_{k\geq 0} c(n,k)F^L_{k}(\eta(t))\big( g(n{-}1)-g(n)\big)+\frac{1}{n}\sum_{k\geq 0} c(n,k) F^L_{k}(\eta(t)) \left( g(k{+}1)-g(n) \right)  \bigg)\\
 -\frac{1}{L-1}c(n,n)\bigg(\frac{n+1}{n} \left( g(n{+}1)-g(n) \right)+ \frac{n-1}{n} \big( g(n{-}1)-g(n)\big)\bigg)  \ ,
\end{multline}
and using the symmetry of the kernel, it is reduced to the following form
\begin{multline}\label{genNsimple}
	\hat{\Lcal}^L_{\eta(t)}  g(n)=\frac{L}{L-1}\sum_{k \geq 1} c(k,n)F^L_{k}(\eta(t)) \big(g(n+1)+g(n-1)-2g(n)\big)\\
   {+}\frac{L}{L{-}1} \frac{1}{n}\sum_{k\geq 0} c(n,k)F^L_{k}(\eta(t))\big( g(k{+}1)-g(n-1)\big)\\
 -\frac{1}{L-1}c(n,n)\bigg(\frac{n+1}{n} \left( g(n{+}1)-g(n) \right)+ \frac{n-1}{n} \big( g(n{-}1)-g(n)\big)\bigg)  \ .
\end{multline}
Note that the process $(\W^L(t), t\geq 0)$ is itself not a Markov process, since its generator depends also on the state of the configuration $\eta (t)$. 
Based on Theorem \ref{thmfactorization}, we have that for each $n\in \N$, in the limit $L\to \infty$, \eqref{genN} converges {pointwise} to a time-inhomogeneous generator 
\begin{align}\label{limgen}
	\hat{\Lcal}_t g(n)=\mu_n(t)\big(g(n{+}1)-g(n)\big)&+ \frac{n{-}1}{n} \mu_n(t) \big( g(n{-}1)-g(n)\big)\nonumber \\&\quad+ \frac{1}{n}\sum_{k\geq 1} c(n,k{-}1) f_{k-1}(t) \left( g(k){-}g(n) \right) \ .
\end{align}
This generator describes a birth-death process with time-dependent birth and death rates $\mu_n (t)$ and $\frac{n-1}{n}\mu_n (t)$ as given in \eqref{birthrate}, and with additional long-range jumps when the tagged particle changes position. It is proved that the occupation number on the tagged particle position actually converges to this death-process as $L\to\infty$.

\begin{thm}\label{mainthm}
	Consider a tagged particle process with generator \eqref{genN} on the complete graph with symmetric  rates (\ref{kernel}) and initial conditions satisfying \eqref{initialcon0}, \eqref{initialcon0b}, \eqref{initialcon0c2} and \eqref{initialcon0d}. In the thermodynamic limit \eqref{thermo}, for any $\rho >0$, 
	\[
	\big( \W^L(t) : t   \geq 0\big) \to \big( \hat{\W}(t): t   \geq 0\big) \quad\text{weakly on }   D_{[0,\infty)}(E) ,
	\]
	where $\big( \hat{\W}(t): t\geq 0\big)$ is a time-inhomogeneous Markov process on $\N$ with generator $\hat{\Lcal}_t$ \eqref{limgen}.
\end{thm}
If we consider the  the empirical mass processes
\[
t\mapsto P_k^L (\eta (t)):=\frac{1}{N} \sum_{x\in\Lambda} k\delta_{\eta_x (t),k}=\frac{L}{N} kF^L_k(\eta(t)) \in [0,1]\ ,\quad k\geq 1
\]
which counts the fraction of particles that are on sites with $k$ particles, then by Theorem \ref{thmfactorization}, we have that in the thermodynamic limit
\begin{equation}
    (P_k^L (\eta (t)):t\geq 0) \to (p_k(t):t\geq 0) \text{ weakly on } D_{[0,\infty)}(E) 
\end{equation}
where $p(t)=(p_k(t):k\in\N)$ is the unique global solution to the mean-field equation 
\begin{align}\label{sbmfe}
	\frac{dp_{k}(t)}{dt} &= \mu_{k-1}(t)p_{k-1}(t) +  \frac{k}{k+1}\mu_{k+1}(t) p_{k+1}(t) +\sum_{l\geq 1}\frac{1}{l} c(l,k-1) f_{k-1}(t)p_{l}(t) \nonumber\\
    &\qquad  {-\Big( \mu_k (t){+}\frac{k-1}{k}\mu_{k}(t){+ }\frac{1}{k}\sum_{l\geq 1} c(k,l-1) f_{l-1}(t)\Big)\, p_k(t)}\ ,\qquad k\geq 2\nonumber\\
    \frac{dp_1(t)}{dt}
    &=   {\frac{1}{2}\mu_{2} (t)\, p_{2}(t)+ \sum_{l\geq 1}\frac{1}{l}c(l,0) f_{0}(t) p_{l}(t)    -2\mu_1(t) p_{1}(t) }
\end{align}

Notice that equation \eqref{sbmfe} is the master equation corresponding to the generator \eqref{limgen}. Therefore, based on Theorem \ref{mainthm}, the law of large numbers for the size-biased empirical measures characterizes the behaviour of the occupation number of the tagged particle position in the thermodynamic limit.

\subsection{Proof of Theorem \ref{mainthm}}

\subsubsection{Bounds on moments}
In the following, we denote the  $n$-th moment of the process $\W^L(t)$ by
\begin{equation}\label{hatmom}
	\hat{m}^L_n(t):=\E^L \big[ (\W^L (t))^n \big] =\E^L\big[\left(\eta_{X(t)}(t)\right)^n\big]\ .
\end{equation}
Similarly to Proposition \ref{gwall}, we can establish the following (rough) bounds on the moments of this process.

\begin{prop}\label{gwallN}
    Assume that the sequence  $\big( m^L_{2n+1}(0)\big)_{L\geq 2}$ is bounded for some integer \(n\geq 2\). Then, for each $T>0$ there exists a constant   {$C(n,T)>0$} independent of $L$ such that
	\begin{equation}\label{eq:gwallN}
  	\hat{m}^L_n (t) \leq \hat{m}^L_n(0) e^{  C(n,T)t}\quad\mbox{for all }t\leq T\mbox{ and }L\geq   {2}\ .
	\end{equation}
\end{prop}
\begin{proof}
	Applying the generator \eqref{genNsimple} to the function $g(l)=l^n$ for $n\in \N$, we get for some polynomial $q_{n-2}$  of degree $n-2$
	\begin{multline}\label{genhmn}
		\frac{ d\hat{m}^L_n(t)}{d t}=\E^L \left[\hat{\Lcal}^L_{\eta(t)} (\W^L(t))^n\right] =\frac{L}{L-1} \E^L \left[\sum_{k \geq 1} c(k,\W^L(t))F^L_{k}(\eta(t)) q_{n-2}(\W^L(t))\right] \\
		+  \frac{L}{L-1}\E^L \left[\frac{1}{\W^L(t)}\sum_{k\geq 0} c(\W^L(t),k) F^L_{k}(\eta(t)) \Big( (k+1)^n-(\W^L(t)-1)^n\Big)\right]\\
		-\frac{1}{L-1} \E^L \left[\frac{c(\W^L(t),\W^L(t))}{\W^L(t)}\left( (\W^L(t)+1)p_{n-1}^+(\W^L(t))-(\W^L(t)-1)p_{n-1}^+(\W^L(t)-1) \right) \right] 
	\end{multline}
	where the polynomial \(p^+_{n-1}(x):=(x+1)^n-x^n\) is of degree \(n-1\). 
	Since the functions $  {l}\mapsto lp_{n-1}^+(  {l})$ are increasing for all $n\in \N$, the last  line in \eqref{genhmn} is negative and therefore, we have for some positive constant $C_n$ and all \(t\in [0,T]\) 
	\begin{align*}
		\frac{ d\hat{m}^L_n(t)}{d t}&\leq \frac{L}{L-1} C\E^L \left[ \sum_{k \geq 1} k^{\mu}F^L_{k}(\eta(t)) (\W^L(t))^{\nu}  {q}_{n-2}(\W^L(t)) + \sum_{k \geq 1} k^{\nu}F^L_{k}(\eta(t)) (\W^L(t))^{\mu}  {q}_{n-2}(\W^L(t))\right]\\& +  \frac{L}{L-1}C\E^L \left[\sum_{k\geq 0} k^{\mu}(k+1)^{n}F^L_{k}(\eta(t))(\W^L(t))^{\nu-1} +\sum_{k\geq 0}k^{\nu} (k+1)^{n}F^L_{k}(\eta(t))(\W^L(t))^{\mu-1} \right]\\&
    \leq  C_n\E^L \left[ \sum_{k \geq 1} k^{\mu}F^L_{k}(\eta(t)) (\W^L(t))^{\nu+n-2} + \sum_{k \geq 1} k^{\nu}F^L_{k}(\eta(t)) (\W^L(t))^{\mu+n-2}\right]\\& +   C_n\E^L \left[\sum_{k\geq 0} k^{\mu+n}F^L_{k}(\eta(t))(\W^L(t))^{\nu-1} +\sum_{k\geq 0} k^{\nu+n}F^L_{k}(\eta(t))(\W^L(t))^{\mu-1} \right] 
    \\&  \leq
   \rho C_n \left( (m_{n+1}(t))^{\frac{2-\nu}{n}} (\hat{m}_n(t))^{1-\frac{2-\nu}{n}}+(m_{n+1}(t))^{\frac{2-\mu}{n}} (\hat{m}_n(t))^{1-\frac{2-\mu}{n}}\right)\\&+\rho C_n\left((m_{2n+1}(t))^{\frac{\mu+n-1}{2n}} (\hat{m}_n(t))^{\frac{n-\mu-1}{2n}}+(m_{2n+1}(t))^{\frac{\nu+n-1}{2n}}(\hat{m}_n(t))^{\frac{n-\nu+1}{2n}}\right)\\ & \leq C(n,T) \hat{m}_n(t) \ , 
	\end{align*}
    where in the penultimate step we used H\"older's inequality and at the last one, that $m^{L}_{n+1}(t)\leq m^{L}_{2n+1}(t)$, Proposition  \ref{gwall} and $\hat{m}^{L}_n(t)\geq 1$. 
 The result then follows by Gronwall's Lemma.
\end{proof}

\begin{remark}
In Proposition \ref{gwallN}, it suffices to assume that the sequence $\left(m^L_{r_nn+1}(0)\right)_{L\geq 2}$ is bounded for some integer $n\geq 2,$ where $r_n:=\max\{\frac{n+\mu-1}{n-\nu+1},\frac{n+\nu-1}{n-\mu+1},1 \}\leq 2
$.

\end{remark}



\subsubsection{Tightness}

Based on the methodology developed in Section 3, if we assume uniform in $L$ bounds on the moments  $\tilde{m}^L_n(0)$ and $m^L_{2n+1}(0)$  for sufficiently large $n$, we can establish tightness for the process $W^L(t)$. 

\begin{prop}\label{tight\W}
	Consider the process with generator \eqref{genN} and conditions as in Theorem \ref{mainthm}. Denote by $\mathbb{Q}^L$ the law of the process $t \mapsto \W^L(t)$ on path space $D_{[0,\infty)}(\N)$, which is the image measure of $ {\mathbb{P}}^L$ under the mapping $(\eta,x)\mapsto \eta_x$. Then $\mathbb{Q}^L$ is tight as $L \to \infty$.
\end{prop}

\begin{proof}
  We will follow the  methodology  of proof of Proposition \ref{tightness}. Again, it suffices to show that for all $t\geq 0$
	\begin{equation}
		\label{eq:tight2}
		\lim_{a\to \infty}\limsup_{L\to\infty} {\mathbb{P}}^L \big[ W^L(t)\geq a\big] =0,
	\end{equation}
	 and that for any $\epsilon>0$, $t>0$,
	\begin{equation}
		\label{eq:tightness2}
		\lim_{\delta_0\to 0^+}\limsup_{L\to \infty}\sup_{\delta\leq \delta_0}\sup_{\tau \in \mathfrak{T}_t} {\mathbb{P}}^L\big[ |W^L(\tau+\delta)-W^L(\tau)|>\epsilon\big] =0,
	\end{equation}
 where $\mathfrak{T}_t$ is the set of stopping times satisfying $\tau\leq t.$\\
Relation \eqref{eq:tight2} follows easily from Markov's inequality and Proposition \ref{gwallN},
\[ \mathbb{P}^L \big[ W^L(t)\geq a\big]\leq \frac{\hat{m}_1^L(t)}{a}\leq \frac{\alpha_4e^{C(4,t)t}}{a}\quad\mbox{for all }L\geq 2 \ ,\]
where we used that $\hat{m}^L_1(t)\leq \hat{m}^L_4(t)$ and assumption \eqref{initialcon0d}.
Now fix $\delta_0> 0$, $\tau \in \mathfrak{T}_t$ and consider $\delta<\delta_0$.  By It\^{o}'s formula, we have for all $u>0$
	\begin{equation}
		\label{eq:itoetaMIMt}
  W^L(u+\delta)-W^L(u) =\int_u^{u+\delta} \hat{\Lcal}^L_{\eta(s)} W^L (s)\, ds +M (u+\delta)-M(u)\ ,
	\end{equation}
	where $(M (u) : u\geq 0)$ is a martingale with predictable quadratic variation given by integrating the 'carr\'e du champ' operator, i.e.  
\begin{equation}\label{eq:mart2}
\langle M \rangle(t)=\int_0^t \big(\hat{\Lcal}^L_{\eta(\cdot)}(W^L)^2-2W^L\hat{\Lcal}^L_{\eta(\cdot)}W^L\big) (s)ds\ .
\end{equation}

Using again Markov's inequality in \eqref{eq:tightness2} and replacing $u$ by the bounded stopping time $\tau\leq t,$ we have to bound
	\begin{align}\label{toboundMimN}
		\E^L\Big[\big| W^L(\tau+\delta)-W^L(\tau)\big|\Big]&\leq \E^L \Big[\int \limits_{\tau}^{\tau+\delta}  |\hat{\Lcal}^L_{\eta(s)} W^L(s)|\,ds\Big] +  \E^L \big[ \left(M (\tau+\delta)-  M(\tau) \right)^2\big]^{1/2} \ \nonumber \\
        &\leq \sqrt[3]{\delta} \left( \int \limits_{0}^{t+\delta}\E^L \Big[  |\hat{\Lcal}^L_{\eta(s)} W^L(s)|^{3/2}\Big]\,ds \right)^{2/3} {+} \E^L \big[  \langle M \rangle (\tau{+}\delta){-} \langle M \rangle(\tau)\big]^{1/2}  
	\end{align}
	where we used Cauchy-Schwarz, H\"older's inequality and the stopping time theorem for the martingale $M^2 (t)-\langle M \rangle(t)$. \\

\noindent
We will first compute $\hat{\Lcal}^L_{\eta} W^L$, 
\begin{equation}
\hat{\Lcal}^L_{\eta} W^L=\frac{L}{L{-}1} \frac{1}{W^L}\sum_{k\geq 0} c(W^L,k)F^L_{k}(\eta)( k-W^L+2)
 -\frac{1}{L-1}c(W^L,W^L) \frac{2}{W^L}
\end{equation}
Thus, we have for all $\eta \in E_{L,N}$
\begin{align*}
|\hat{\Lcal}^L_{\eta} W^L|&\leq 2C \left(  \sum_{k}  k^{\nu+1}F^L_k(\eta) (W^L)^{\mu-1}+\sum_{k}  k^{\mu+1}F^L_k(\eta) (W^L)^{\nu-1} \right) \\&+6C  \left(  \sum_{k}  k^{\nu}F^L_k(\eta) (W^L)^{\mu}+\sum_{k}  k^{\mu}F^L_k(\eta) (W^L)^{\nu} \right)  \\&+
\frac{4C}{L-1} (W^L)^{\mu+\nu-1} 
\end{align*}
which for $\mu,\nu\leq 2$ and $\mu+\nu\leq 3$ implies that for some constant $C_{\rho}$ independent of $L$
\begin{equation}
    \E^L \Big[  |\hat{\Lcal}^L_{\eta(s)} W^L(s)|^{3/2}\Big] \leq C_{\rho} \left( (m^L_6(t))^{5/8} (\hat{m}^L_4(t))^{3/8}+ (m^L_7(t))^{1/4} (\hat{m}^L_4(t))^{3/4} + \hat{m}^L_3(t)\right)\ .
\end{equation}
This is uniformly bounded in $L$ on \([0,t+\delta_0]\) by Propositions \ref{gwall} and \ref{gwallN} ans assumptions \eqref{initialcon0c2} and \eqref{initialcon0d}, i.e.
\begin{equation}\label{final1}
 \sup_L \E^L \Big[  |\hat{\Lcal}^L_{\eta(s)} W^L(s)|^{3/2}\Big]  \leq D_{t+\delta_0} \text{ for all} \; s\in [0,t+\delta_0]
\end{equation}
for some positive constant $D_{t+\delta_0}$ (independent of $L$).
Then, to control the last term of \eqref{toboundMimN}, it suffices to bound (uniformly in $L$) the 'carr\'e du champ' operator, for which we have 
\begin{align*}
   0\leq \hat{\Lcal}^L_{\eta}(W^L)^2-2W^L\hat{\Lcal}^L_{\eta}W^L&=\frac{L}{L-1} \frac{1}{W^L} \sum_k c(W^L,k)F^L_k(\eta)(k^2+(W^L)^2-2k(W^L-1))\\ & -\frac{2}{L-1}c(W^L,W^L) \\& \leq 2 \frac{1}{W^L} \sum_k c(W^L,k)F^L_k(\eta)(k^2+(W^L)^2 ) \\&\leq 2 \left(  \sum_k k^{\mu+2}F^L_k(\eta) (W^L)^{\nu-1} + \sum_k k^{\nu+2}F^L_k(\eta) (W^L)^{\mu-1}  \right) \\&+ 2\left( \sum_k k^{\mu}F^L_k(\eta) (W^L)^{\nu+1}+\sum_k k^{\nu}F^L_k(\eta) (W^L)^{\mu+1}\right) \ . 
\end{align*}
Thus, using that $\mu,\nu \leq 2$, we find for some positive constant $C'_{\rho}$
\begin{equation}
  \E^L \Big[  |\hat{\Lcal}^L_{\eta(s)}(W^L(s))^2-2W^L(s)\hat{\Lcal}^L_{\eta(s)}W^L(s)|^{8/7} \Big]  \leq C'_{\rho} \left( (m_6^L(t))^{5/7}(\hat{m}_4^L(t))^{2/7}+(m_9^L(t))^{1/7}(\hat{m}_4^L(t))^{6/7} \right) 
\end{equation}
which, is uniformly bounded in $L$ on \([0,t+\delta_0]\) by Propositions \ref{gwall} and \ref{gwallN}, i.e.
\begin{equation}
 \sup_L \E^L \Big[  |\hat{\Lcal}^L_{\eta(s)}(W^L(s))^2-2W^L(s)\hat{\Lcal}^L_{\eta(s)}W^L(s)|^{8/7}  \Big]  \leq D'_{t+\delta_0} \text{ for all} \; s\in [0,t+\delta_0]
\end{equation}
for some positive constant $D'_{t+\delta_0}$ (independent of $L$).
Therefore, we have 
\begin{align}
\E^L \big[  \langle M\rangle (\tau+\delta)-  \langle M \rangle(\tau)\big] &= \E^L \Big[ \int_{\tau}^{\tau+\delta} \big( \hat{\Lcal}^L_{\eta(\cdot)}(W^L)^2-2W^L\hat{\Lcal}^L_{\eta(\cdot)}W^L\big) (s)ds \Big] \nonumber \\& \leq \delta_0^{1/8}\left( \int  \limits_{0}^{t+\delta} \E^L \Big[  |\hat{\Lcal}^L_{\eta(s)}(W^L(s))^2-2W^L(s)\hat{\Lcal}^L_{\eta(s)}W^L(s)|^{8/7}  \Big] ds\right)^{7/8} \nonumber \\& \leq \delta_0^{1/8} ((t+\delta_0)D'_{t+\delta_0})^{7/8} \ . \label{final2}
\end{align}
Thus, based on estimates \eqref{final1} and \eqref{final2}, relation \eqref{toboundMimN} 
 becomes 
\begin{equation}\label{tight_bound2fin}
     \E^L\Big[\big| W^L(\tau+\delta)-W^L(\tau)\big|\Big] \leq \sqrt[3]{ \delta_0} \left((t+\delta_0) D_{t+\delta_0}\right)^{2/3} + \delta_0^{1/16} ((t+\delta_0)D'_{t+\delta_0})^{7/16}
\end{equation}
which vanishes as $\delta_0\to 0$, finishing the proof.  
\end{proof}
By Prokhorov’s theorem, the tightness result in Proposition \ref{tight\W} implies the existence of sub-sequential limit points of the sequence $(\W^L(t) : t \geq 0)$ in the usual Skorohod topology of weak convergence on path space $D_{[0, \infty )} (\N )$ (see e.g.\ \cite{billingsley2013convergence}, Section 16). We denote the law of any such limit by $\mathbb{Q}$.

\subsubsection{Characterization of the limit process}

In order to identify the limit $\mathbb{Q}$ we need to show that for all $t  \geq 0$ and $g \in C_b (\N )$, 
\begin{equation}\label{eq:marti}
	g(\omega(t))-g(\omega (0))-\int_0^t \hat{\Lcal}_s g(\omega(s))ds \; \text{ is a martingale w.r.t. }\mathbb{Q}\ ,
\end{equation}
where $\omega\in D_{[0,  \infty)} (\N)$ denotes an element in path space.  Together with the uniqueness of the martingale problem associated with $\hat{\Lcal}_t$, this implies convergence of $\mathbb{Q}^L$ and characterizes the limit $\mathbb{Q}$ as the law of the Markov process $(\hat{\W}(t): \; t\geq 0)$  with generator $\hat{\Lcal}_t$ \eqref{limgen}. 
Following the methodology in \cite{SGAKtagged} (Appendix C), we only need to prove that for all $  t\geq 0$
\begin{equation}\label{toshowa}
	\E^L \left[ \left| \int_0^{  {t}} \left( \hat{\Lcal}_{  {s}} g(\W^L({  {s}})) -\hat{\Lcal}^L_{\eta({  {s}})} g(\W^L({  {s}}))\right)d{  {s}} \right| \right]  \to 0\quad\mbox{as }
L\to\infty\ .
\end{equation}
Since the process $t\mapsto\hat{\Lcal}_{  {\eta(t)}}^L g(\W^L(t))$ is bounded in $L^1$-norm on compact time intervals uniformly with respect to $L$, and using the triangle inequality it suffices to prove that


	\begin{equation}\label{conv_to_zero}
		\int_0^{t} \E^L \left[\left| \hat{\Lcal}_{s} g(\W^L({s})) -  \frac{L-1}{L} \hat{\Lcal}_{\eta({s})}^L g(\W^L({s}))\right|\right] d{s} \ \to 0
	\end{equation}
	as $L\to \infty$. 
	Since $g\in C_b (\N )$ 
	and because of conditions \eqref{symmetry} and \eqref{kernel}, we find
	\begin{align*}
		&\left| \hat{\Lcal}_{s} g(\W^L{(s)}) -  \frac{L-1}{L} \hat{\Lcal}_{\eta{(s)}}^L g(\W^L{(s)}) \right|  \\
		 &\qquad \qquad \leq 4||g||_{\infty} \Bigg( \sum_{k \geq 1} c(k,\W^L{(s)})\left|F^L_{k}(\eta{(s)})-f_{k}{(s)}\right| +\frac{c (\W^L(s),W^L(s))}{L}\Bigg)\\
		& \qquad \qquad\leq 4C||g||_{\infty} \Bigg(   (W^L(s))^{\mu} \sum_{k \geq 1} k^{\nu}\left|F^L_{k}(\eta{(s)}){-}f_{k}{(s)}\right| \\& \qquad \qquad
		{+}(W^L(s))^{\nu} \sum_{k \geq 1} k^{\mu}\left|F^L_{k}(\eta{(s)}){-}f_{k}{(s)}\right| {+}\frac{2(\W^L (s))^{\mu+\nu}}{L}\Bigg)  {\ .}
	\end{align*}
	Notice that for all {$s\leq t,$} and $M>0$, since $\mu,\nu\leq 2$, we have
	\begin{align*}
		\E^L & \bigg[ (W^L(s))^{\mu} \sum_{k \geq 1} k^{\nu}\left|F^L_{k}(\eta{(s)}){-}f_{k}{(s)}\right| \bigg] \\ & \qquad \leq  
  \E^L \bigg[ (W^L(s))^{2\mu} \bigg]^{1/2} \E^L\bigg[ \bigg( \sum_{k \geq 1} k^{\nu}\left|F^L_{k}(\eta{(s)}){-}f_{k}{(s)}\right|\bigg)^2 \bigg]^{1/2} \\ & \qquad
        \leq  (\hat{m}^L_4(t))^{1/2} \E^L \bigg[\sum_{k \geq 1} k^{2\nu}\left|F^L_{k}(\eta{(s)}){+}f_{k}{(s)}\right| \sum_{l \geq 1} \left|F^L_{l}(\eta{(s)}){-}f_{l}{(s)}\right|  \bigg]^{1/2}\ .
	\end{align*}
Let  $Y_{\nu}^L(s):=\sum_{k \geq 1} k^{2\nu}\left|F^L_{k}(\eta{(s)}){+}f_{k}{(s)}\right|$. Then 
    \begin{align*}
          \E^L \bigg[Y_{\nu}^L(t) \sum_{k \geq 1} \left|F^L_{k}(\eta{(s)}){-}f_{k}{(s)}\right|  \bigg] \leq   M &\E^L \bigg[ \sum_{l \geq 1} \left|F^L_{l}(\eta{(s)}){-}f_{l}{(s)}\right|  \bigg]  \\ &+  2\rho\sup_{L\geq 2,  s \leq t}\E^L \big[ Y_{\nu}^L(s)\mathbbm{1} \{ Y_{\nu}^L(s)>M\} \big]
    \end{align*}
Defining $Y^L_{\mu}$ analogously, the fact that $\mu+\nu\leq 3$ and Proposition \ref{gwallN} imply that there exists $C_t>0$ independent of $L$ such that  
	\begin{align*}
		\int_0^{t} \E^L &\left[\left| \hat{\Lcal}_s g(\W^L(s)) -  \frac{L-1}{L} \hat{\Lcal}_{\eta(s)} g(\W^L(s))\right|\right] ds  \\&   \leq  4C||g||_{\infty}   \Bigg( C_t\int_0^{t} \Bigg( M \E^L \big[ \sum_{k \geq 1} \left|F^L_{k}(\eta{(s)}){-}f_{k}{(s)}\right|  \big] \\& \qquad \qquad  +  2\rho\sup_{L\geq 2,  s \leq t}\E^L \big[ Y_{\nu}^L(s)\mathbbm{1} \{ Y_{\nu}^L(s)>M\} \big] \Bigg)^{1/2}ds \Bigg) \\& \quad + 4C||g||_{\infty}   \Bigg( C_t\int_0^{t} \Bigg( M \E^L \big[ \sum_{k \geq 1} \left|F^L_{k}(\eta{(s)}){-}f_{k}{(s)}\right|  \big] \\& \qquad \qquad +  2\rho\sup_{L\geq 2,  s \leq t}\E^L \big[ Y_{\mu}^L(s)\mathbbm{1} \{ Y_{\mu}^L(s)>M\} \big] \Bigg)^{1/2}ds \Bigg) \\& \quad +8C||g||_{\infty}  \frac{C_tt}{L} \ .
	\end{align*}
	In the limit $L\to \infty$, based on Lemma  \ref{convsum}, we have for all $M>0$,
	\begin{multline*}
		  {\limsup \limits_{L\to \infty}}\int_0^t \E^L \left[\left| \hat{\Lcal}_s g(\W^L(s)) -  \frac{L-1}{L} \hat{\Lcal}^L_{\eta(s)} g(\W^L(s))\right|\right] ds   \\   \leq  4C\|g\|_{\infty} C_t\sqrt{2\rho}t \bigg( \sup_{L\geq 2,  s \leq t}\E^L \big[ Y_{\mu}^L(s)\mathbbm{1} \{ Y_{\mu}^L(s)>M\} \big] \big)^{1/2} \\ +\sup_{L\geq 2,  s \leq t}\E^L \big[ Y_{\nu}^L(s)\mathbbm{1} \{ Y_{\nu}^L(s)>M\} \big] \big)^{1/2}\bigg)  \ .
	\end{multline*}
    Notice that H\"older's inequality and Proposition \ref{gwall} give that the processes  $\{Y_{\mu}^L(s)\}_{L\geq 2,s\leq t}$ and $\{Y_{\nu}^L(s)\}_{L\geq 2,s\leq t}$ are uniform integrable. Therefore, taking $M\to \infty$  gives \eqref{conv_to_zero}.



\bibliographystyle{amsplain}
\bibliography{ref_new}


\end{document}